\documentclass[12pt]{amsart}
\usepackage{amssymb, amscd, amsmath, amsthm, latexsym, enumerate}
\newtheorem{theorem}{Theorem}
\newtheorem{lemma}[theorem]{Lemma}

\newtheorem*{cor}{Corollary}

\begin{document}

\title{Finiteness conditions in covers of Poincar\'e duality spaces}

\author{Jonathan A. Hillman }
\address{School of Mathematics and Statistics\\
     University of Sydney, NSW 2006\\
      Australia }

\email{jonathan.hillman@sydney.edu.au}

\begin{abstract}
A closed 4-manifold (or, more generally, a finite $PD_4$-space) 
has a finitely dominated infinite regular covering space if and only if 
either its universal covering space is finitely dominated or 
it is finitely covered by the mapping torus of a self homotopy equivalence 
of a $PD_3$-complex.
\end{abstract}



\maketitle

A space $X$ is a {\it Poincar\'e duality space\/} if it has
the homotopy type of a cell complex which satisfies Poincar\'e 
duality with local coefficients (with respect to some orientation character
$w:\pi=\pi_1(X)\to\{\pm1\}$).
It is {\it finite\/} if the singular chain complex of the universal cover
$\widetilde{X}$ is chain homotopy equivalent to a finite free 
$\mathbb{Z}[\pi]$-complex.
(The $PD$-space $X$ is homotopy equivalent to a Poincar\'e duality complex 
$\Leftrightarrow$ it is finitely dominated $\Leftrightarrow$
$\pi$ is finitely presentable.
See \cite{Bd}.)
Closed manifolds are finite $PD$-complexes.
The more general notion arises naturally
in connection with Poincar\'e duality groups \cite{Br},
and in considering covering spaces of manifolds \cite{HK}.

In this note we show that finiteness hypotheses in two theorems 
about covering spaces of $PD$-complexes may be relaxed.
Theorem 5 extends a criterion of Stark to all Poincar\'e duality groups.
The main result is Theorem 6,
which characterizes finite $PD_4$-spaces with finitely dominated 
infinite regular covering spaces.

\section{some lemmas}

Let $X$ be a $PD_n$-space with fundamental group $\pi$.
Let $\beta_i(X;\mathbb{Q})=\mathrm{dim}_{\mathbb{Q}}H_i(X;\mathbb{Q})$
and $\beta_i^{(2)}(X)=\mathrm{dim}_{\mathcal{N}(\pi)}H_i(X;\mathcal{N}(\pi))$
be the $i$th rational Betti number and $i$th $L^2$ Betti number of $X$,
respectively.

\begin{lemma}
Let $X$ be a $PD_n$-space with fundamental group $\pi$.
Then $\Sigma\beta_i(X;\mathbb{Q})<\infty$ and
$\Sigma\beta_i^{(2)}(X)$ $<\infty$.
If $X$ is finite then 
$\chi(X)=\Sigma(-1)^i\beta_i(X;\mathbb{Q})=\Sigma(-1)^i\beta_i^{(2)}(X)$.
\end{lemma}

\begin{proof}
Since $X$ is a $PD_n$-space and homology commutes with 
direct limits of coefficient modules so does cohomology.
Therefore the singular chain complex of $\widetilde{X}$ 
is chain homotopy equivalent over $\mathbb{Z}[\pi]$
to a finite projective complex $P_*$, 
by the Brown-Strebel finiteness criterion \cite{Br75}.
Hence $H_i(X;\mathbb{Q})=
H_i(\mathbb{Q}\otimes_{\mathbb{Z}[\pi]}{P_*})$
and $H_i(X;\mathcal{N}(\pi))=H_i(\mathcal{N}(\pi)\otimes_{\mathbb{Z}[\pi]}{P_*})$.
The first assertion follows immediately.
The proof of the $L^2$-Euler characteristic formula for finite complexes
given in \cite{Lu} is entirely homological, 
and requires only that $C_*$ be chain homotopy
equivalent to a finite free complex.
\end{proof}

If the strong Bass conjecture holds for $\pi$ then 
the $L^{(2)}$-Euler characteristic formula holds
even if $X$ is not finite \cite{Ec}.
 
The following lemma is essentially from \cite{Fa}.
We shall use it in conjunction with universal coefficient spectral
sequences.

\begin{lemma}
Let $G$ be a group and $k$ be $\mathbb{Z}$ or a field, 
and let $A$ be a $k[G]$-module
which is free of finite rank $m$ as a $k$-module.
Then $Ext^q_{k[G]}(A,k[G])\cong(H^q(G;k[G]))^m$ for all $q$.
\end{lemma}

\begin{proof}
Let $(g\phi)(a)=g.\phi(g^{-1}a)$ for all $g\in{G}$ and $\phi\in{Hom_k(A,k[G])}$.
Let $\{\alpha_i\}_{1\leq{i}\leq{m}}$ be a basis for $A$ as a free $k$-module,
and define a map $f:Hom_k(A,k[G])$ $\to{k[G]}^m$ 
by $f(\phi)=(\phi(\alpha_1),\dots,\phi(\alpha_m))$
for all $\phi\in{Hom_k(A,k[G])}$.
Then $f$ is an isomorphism of left $k[G]$-modules.
The lemma now follows, since
$Ext^q_{k[G]}(A,k[G])\cong{H^q}(G;Hom_k(A,k[G]))$.
(See Proposition III.2.2 of \cite{Br}.)
\end{proof}

\begin{lemma}
If $H^q(G;\mathbb{Z}[G])$ is $0$ (respectively,
finitely generated as an abelian group) for 
all $q\leq{q_0}$ and $B$ is a $\mathbb{Z}[G]$-module 
which is finitely generated as an abelian group 
then $Ext^q_{\mathbb{Z}[G]}(B,\mathbb{Z}[G])$ is $0$
(respectively, finitely generated as an abelian group) for all $q\leq{q_0}$.
\end{lemma}

\begin{proof}
Let $T$ be the $\mathbb{Z}$-torsion submodule of $B$,
and let $H$ be the kernel of the action of $G$ on $T$.
Then $T$ is a finite $\mathbb{Z}[G/H]$-module,
and so is a quotient of a finitely generated free $\mathbb{Z}[G/H]$-module $A$.
Let $A_1$ be the kernel of the projection from $A$ to $T$.
Clearly $A$ and $A_1$ are $\mathbb{Z}[G]$-modules 
which are free of (the same) finite rank as abelian groups.
We now apply the long exact sequence of 
$Ext^*_{\mathbb{Z}[G]}(-,\mathbb{Z}[G])$ together with Lemma 2
to the short exact sequences
\[
0\to{A_1}\to{A}\to{T}\to0
\] 
and 
\[0\to{T}\to{B}\to{B/T}\to0.
\]
\end{proof}

\newpage
\section{virtual poincar\'e duality groups}

Stark has shown that a finitely presentable group $G$ of finite virtual 
cohomological dimension is a virtual Poincar\'e duality group 
if and only if it is the fundamental group of a closed $PL$ manifold $M$ 
whose universal cover $\widetilde{M}$ is homotopy finite \cite{St95}.
The main step in showing the sufficiency of the latter condition involves 
showing first that $G$ is of type $vFP$, and is established in \cite{St96}.
If $G_1$ is an $FP$ subgroup of finite index in $G$
then $B=K(G_1,1)$ is finitely dominated.
Hence on applying the Gottlieb-Quinn Theorem to the
fibration $\widetilde{M}\to{M_1}\to{B}$ of the associated 
covering space $M_1$ it follows that 
$\widetilde{M}$ and $B$ are Poincar\'e duality complexes.
In particular, $G_1$ is a Poincar\'e duality group.

There are however Poincar\'e duality groups in every dimension $n\geq4$ which
are not finitely presentable.
We shall give an analogue of Stark's sufficiency result for such groups,
using an algebraic criterion instead of the Gottlieb-Quinn Theorem.
In the next two results we shall assume that
$M$ is a $PD_n$-space with fundamental group $\pi$, 
$M_\nu$ is the covering space associated to
a normal subgroup $\nu$ of $\pi$, $G=\pi/\nu$ and
$k$ is $\mathbb{Z}$ or a field.

\begin{lemma}
Suppose that $H_p(M_\nu;k)$ is finitely generated for all $p\leq[n/2]$.
Then $H_p(M_\nu;k)$ is finitely generated for all $p$ if and only if 
$H^q(G;k[G])$ is finitely generated as a $k$-module for $q\leq[(n-1)/2]$,
and then $H^q(G;k[G])$ is finitely generated as a $k$-module for all $q$.
If $H^s(G;k[G])=0$ for $s<q$ then $H_{n-s}(M_\nu;k)=0$ for $s<q$
and $H_{n-q}(M_\nu;k)\cong{H^q}(G;k[G])$.
\end{lemma}

\begin{proof} 
Let $E_2^{pq}=Ext^q_{k[G]}(H_p(M;k[G]),k[G])\Rightarrow{H^{p+q}}(M;k[G])$
be the Universal Coefficient spectral sequence for the equivariant 
cohomology of $M$.
Then $E_2^{pq}=Ext^q_{k[G]}(H_p(M_\nu;k),k[G])$, while
$H^{p+q}(M;k[G])\cong{H_{n-p-q}}(M_\nu;k)$, by
Poincar\'e duality for $M$. 

If $H^q(G;k[G])$ is finitely generated for $q\leq[(n-1)/2]$
then $E_2^{pq}$ is finitely generated for all $p+q\leq[(n-1)/2]$, 
by Lemmas 2 and 3.
Hence $H_p(M_\nu;k)$ is finitely generated for all $p\geq{n}-[(n-1)/2]$,
and hence for all $p$.
Conversely, if this holds and $H^s(G;k[G])$ is finitely generated for $s<q$
then $E_r^{ps}$ is finitely generated for all $p\geq0$, $r\geq2$ and $s<q$.
Since $H^q(M;k[G])\cong{H_{n-q}}(M_\nu;k)$ is finitely generated
as a $k$-module
it follows that $H^q(G;k[G])$ is finitely generated as a $k$-module.
Hence $H^q(G;k[G])$ is finitely generated for all $q$.

The final assertion is an immediate consequence of duality and the universal
coefficient spectral sequence.
\end{proof}

\begin{theorem}
If  $H_p(M_\nu;k)$ is finitely generated for all $p$ then $G$ is $FP_\infty$ 
over $k$ and $H^s(G;k[G])\not=0$ for some $s\leq{n}$.
If moreover $k=\mathbb{Z}$ and $v.c.d.G<\infty$ 
then $G$ is virtually a $PD_r$-group, for some $r\leq{n}$.
\end{theorem}

\begin{proof} 
Let $C_*(\widetilde{M})$ be the equivariant chain complex of the 
universal covering space $\widetilde{M}$.
Since $M$ is a $PD_n$-space $C_*(\widetilde{M})$ is chain homotopy 
equivalent to a finite projective $\mathbb{Z}[\pi]$-complex.
Hence $C_*(M_\nu;k)=k[G]\otimes_{\mathbb{Z}[\pi]}C_*(\widetilde{M})$ 
is chain homotopy 
equivalent to a finite projective $k[G]$-complex.
The arguments of \cite{St96} apply equally well with 
coefficients $k$ a field (instead of $\mathbb{Z}$), 
and thus the hypotheses of Lemma 4 imply that
$G$ is $FP_\infty$ over $k$.

If $v.c.d.G<\infty$ we may assume without loss of generality 
that $c.d.G<\infty$, and so $G$ is $FP$.
Since $H_q(M_\nu;\mathbb{Z}))$ is finitely generated for all $q$ 
the groups $H^s(G;\mathbb{Z}[G])$ are all finitely generated, 
and since $H_0(M_\nu;\mathbb{Z})=\mathbb{Z}$ we must have 
$H^s(G;\mathbb{Z}[G])\not=0$ for some $s\leq{n}$, by Lemma 4.
Then $G$ is a $PD_s$-group, by Theorem 3 of \cite{Fa}.
\end{proof}

A finitely generated group $G$ is a {\it weak $PD_r$-group\/} if 
$H^r(G;\mathbb{Z}[G])\cong\mathbb{Z}$ and $H^q(G;\mathbb{Z}[G])=0$ 
 for $q\not=r$.
Theorem 5 complements the main result of \cite{HK}, 
in which it is shown that if the $\mathbb{Z}[\nu]$-chain complex 
$C_*(\widetilde{M_\nu})=C_*(\widetilde{M})|_\nu$ has finite $[n/2]$-skeleton 
and $G$ is a weak $PD_r$-group then $M_\nu$ is a $PD_{n-r}$-space.

For each $n\geq2$ and $k\geq\binom{n+1}2$
there are weak $PD_k$-groups which act freely 
and cocompactly on $S^{2n-1}\times\mathbb{R}^k$, 
but which are not virtually torsion-free \cite{FS}. 
Thus if $r\geq6$ weak $PD_r$-groups need not be virtual $PD_r$-groups,
and so the other conditions in Theorem 5
do not imply that $v.c.d.G<\infty$,
in general.
Weak $PD_1$-groups have two ends, and so are virtually $\mathbb{Z}$, 
while $FP_2$ weak $PD_2$-groups are virtual $PD_2$-groups \cite{Bo}.
Little is known about the intermediate cases $r=3,4$ or 5.
In particular,
it is not known whether a group $G$ of type $FP_\infty$
such that $H^3(G;\mathbb{Z}[G])\cong\mathbb{Z}$ must be a virtual $PD_3$-group.
(The fact that local homology manifolds which are homology 2-spheres 
are standard may be some slight evidence for this being true.)

Stark's argument for realization in the finitely presentable case
can be adapted to show that any virtual $PD_n$-group 
acts freely on a 1-connected homotopy finite complex,
with quotient a $PD_m$-space for some $m\geq{n}$. 
However finite presentability is needed in order to obtain a free 
{\it cocompact\/} action on a $1$-connected complex.
A natural converse to Theorem 5 (analogous to Stark's realization result)
might be that every virtual $PD$ group $G$ acts freely 
and cocompactly on some connected manifold $X$ 
with $H_q(X;\mathbb{Z})$ finitely generated for all $q$.
It would suffice to show that $G\cong\pi/\nu$ where $\pi$ 
is a finitely presentable $vPD$-group and $\nu$ is a normal subgroup 
such that $H_*(\nu;\mathbb{Z})$ is finitely generated.
For there is a closed PL manifold $M$ with $\pi_1(M)\cong\pi$ 
and $\widetilde{M}$ homotopy finite, by Stark's result. 
The quotient group $G$ acts freely and cocompactly on $M_\nu$,
and a spectral sequence argument shows that
$H_*(M_\nu;\mathbb{Z})$ is finitely generated.

\section{finitely dominated covering spaces of $PD_4$-spaces}

Let $M$ be a $PD_4$-space with fundamental group $\pi$, 
and suppose that $M$ has a finitely dominated infinite regular 
covering space $M_\nu$.
Then $\nu=\pi_1(M_\nu)$ is finitely presentable and $\pi/\nu$ has one or two
ends.
In \cite{Hi} we showed that if $\pi/\nu$ has two ends then
$M$ is the mapping torus of a self homotopy equivalence of a $PD_3$-complex,
while if $\pi/\nu$ has one end and $\nu$ is $FP_3$ 
then either the universal covering space $\widetilde{M}$ is contractible 
or homotopy equivalent to $S^2$.
We shall show here that the hypothesis that $\nu$ be $FP_3$
is redundant if $M$ is a closed 4-manifold, 
or more generally if $M$ is a finite $PD_4$-space.

The results from \cite{Hi} used in the next theorem
were originally formulated in terms of $PD_4$-complexes.
The arguments given in \cite{Hi} apply equally well to $PD_4$-spaces,
since they need only the $L^{(2)}$-Euler characteristic formula 
of Lemma 1 above.

\begin{theorem}
Let $M$ be a finite $PD_4$-space with fundamental group $\pi$,
and let $\nu$ be an infinite normal subgroup of $\pi$ such that
$G=\pi/\nu$ has one end and
the associated covering space $M_\nu$ is finitely dominated.
Then $G$ is of type $FP_\infty$ and $M$ is aspherical.
\end{theorem}

\begin{proof}
Let $k$ be $\mathbb{Z}$ or a field.
Then $G$ is of type $FP_\infty$ and $H^q(G;k[G])$ is finitely generated 
as a $k$-module for all $q$, 
by Lemma 4 and Theorem 5. 
Moreover $Ext^q_{k[\pi]}(H_p(M_\nu;k),k[\pi])=0$ for $q\leq1$ and all $p$,
since $G$ has one end, and so $H_q(M_\nu;k)=0$ for $q\geq3$.
In particular, $H^2(G;\mathbb{Z}[G])\cong{H_2(M_\nu;\mathbb{Z})}$ is 
torsion-free, and so is a free abelian group of finite rank.

We may assume that $M_\nu$ is not acyclic and $G$ 
is not virtually a $PD_2$-group,
by Theorem 3.9 of \cite{Hi}.
Therefore $H^2(G;k[G])=0$ for all $k$,
by the main result of \cite{Bo}.
Hence $H_2(M_\nu;\mathbb{F}_p)=0$ for all primes $p$,
so $H_1(M_\nu;\mathbb{Z})$ is torsion-free and nonzero.
Therefore $H^s(G;\mathbb{Z}[G])=H_{4-s}(M_\nu;\mathbb{Z})=0$ for $s<3$
and $H^3(G;\mathbb{Z}[G])\cong H_1(M_\nu;\mathbb{Z})=\nu/\nu'$ 
is a nontrivial finitely generated abelian group.
Therefore $\nu/\nu'\cong{H^3(G;\mathbb{Z}[G])}\cong\mathbb{Z}$ \cite{Fa}.

Thus we may assume that $M_\nu$ is an homology circle.
Let $\tilde{G}=\pi/\nu'$ and let $t\in\tilde{G}$ represent a generator of the
infinite cyclic group $\nu/\nu'$.
Let $M_\nu'$ be the covering space associated to the subgroup $\nu'$.
Since $M_\nu$ is finitely dominated a Wang sequence argument
shows that $H_q(M_\nu';k)$ is a finitely generated $k[t,t^{-1}]$-module on
which $t-1$ acts invertibly, for all $q>0$.
Then $H_q(M_\nu';\mathbb{F}_p)$ is finitely generated for all primes $p$ 
and all $q>0$.
Now $H^s(\tilde{G};k[\tilde{G}])=0$ for all $k$ and all $s<4$,
by a Lyndon-Hochschild-Serre spectral sequence argument.
Therefore $H_q(M_\nu';\mathbb{F}_p)=0$ for all primes $p$ and all $q>0$,
by Lemma 4.
Nontrivial finitely generated $\mathbb{Z}[t,t^{-1}]$-modules have nontrivial
finite quotients, and so we may conclude that $M_\nu'$ is acyclic. 

Since $M$ is a $PD_4$-space $C_*(\widetilde{M})$
is chain homotopy equivalent to a finite projective 
$\mathbb{Z}[\pi]$-complex $C_*$. 
Thus $D_*=\mathbb{Z}\otimes_{\mathbb{Z}[\nu']}C_*$ is a finite projective 
$\mathbb{Z}[\tilde{G}]$-complex, and is a resolution of $\mathbb{Z}$.
Therefore $\tilde{G}$ is a $PD_4$-group.
(In particular, we see again that $G=\tilde{G}/(\nu/\nu')$ is $FP_\infty$.)

Since $\nu/\nu'$ is a torsion-free abelian normal subgroup of $\tilde{G}$ 
the group ring $\mathbb{Z}[\tilde{G}]$ has a flat extension $R$, 
obtained by localising with respect to the nonzero elements of 
$\mathbb{Z}[t,t^{-1}]$, such that 
$R\otimes_{\mathbb{Z}[\tilde{G}]}\mathbb{Z}=0$.
(See page 23 of \cite{Hi} and the references there.)
Hence $R\otimes_{\mathbb{Z}[\tilde{G}]}D_*$ is a contractible complex 
of finitely generated projective $R$-modules.

We may in fact assume that $C_*$ is a finite {\it free} 
$\mathbb{Z}[\pi]$-complex,
since $M$ is a {\it finite} $PD_4$-space.
It follows that $\chi(M)=\chi(R\otimes_{\mathbb{Z}[\tilde{G}]}D_*)=0$.
Since $\nu$ is an infinite $FP_2$ normal subgroup of $\pi$ and $\pi/\nu$ 
has one end $\beta_1^{(2)}(\pi)=0$ and $H^s(\pi;\mathbb{Z}[\pi])=0$ 
for $s\leq2$.
Therefore $M$ is aspherical, by Corollary 3.5.2 of \cite{Hi}.
\end{proof}

With this result we may now reformulate Theorem 3.9 of \cite{Hi}
as follows.

\begin{cor}
A finite $PD_4$-space $M$ has a finitely dominated infinite
regular covering space if and only if either $M$ is aspherical, 
or $\widetilde{M}\simeq{S^2}$,
or $M$ has a $2$-fold cover which is homotopy equivalent to the mapping torus 
of a self-homotopy equivalence of a $PD_3$-complex.
If $M$ has a finitely dominated regular covering space 
and is not aspherical it is a $PD_4$-complex.
\end{cor}

\begin{proof}
Only the final sentence needs any comment.
If $\widetilde{M}\simeq{S^2}$ then $\pi_1(M)$ is
virtually a $PD_2$-group and so is finitely presentable. 
This is also clear if $M$ has a $2$-fold cover which is the mapping torus 
of a self-homotopy equivalence of a $PD_3$-complex.
Thus in each case $M$ is a $PD_4$-complex.
\end{proof}

There are $PD_n$ groups of type $FF$ which are not finitely presentable, 
for each $n\geq4$ \cite{Da}.
The corresponding $K(G,1)$ spaces are aspherical finite $PD_n$-spaces 
which are not $PD_n$-complexes.

The hypothesis that $M$ be finite is used 
only in the final paragraph of the proof of Theorem 6,
in the appeal to Corollary 3.5.2 of \cite{Hi}
and in the calculation of $\chi(M)$.
(If we assumed instead that $v.c.d.G<\infty$ then 
we could use multiplicativity of the Euler characteristic
to show that $\chi(M)=0$.) 

A more substantial issue is that the argument for Theorem 6 
does not appear to extend to the case when $\nu$ is an ascendant subgroup 
of $\pi$, 
as considered in \cite{Hi08} (where the $FP_3$ condition is also used).
Is there an argument along the following lines?
Let $C_*$ be a finite projective $\mathbb{Z}[\pi]$-complex
with $H_0(C_*)\cong\mathbb{Z}$ and $H_1(C_*)=0$.
Show that $Hom_{\mathbb{Z}[\pi]}(H_2(C_*),\mathbb{Z}[\pi])=0$
if $[\pi:\nu]=\infty$ and $C_*|_\nu$ is chain homotopy equivalent 
to a finite projective $\mathbb{Z}[\nu]$-complex.
If so, the proofs of Theorem 3.9 of \cite{Hi} and Theorem 6 of \cite{Hi08}
would apply, without needing to assume that $\nu$ is $FP_3$ 
or that $M$ is finite.

\section{$PD_4$-complexes with $\pi_3$ finitely generated}

We conclude with an alternative characterization of $PD_4$-complexes 
as in the Corollary to Theorem 6.
Recall that there is a natural exact sequence of left 
$\mathbb{Z}[\pi]$-modules
\[
H_4(\widetilde{X};\mathbb{Z})\to\Gamma_W(\Pi)\to\pi_3(X)\to{H_3(\widetilde{X};\mathbb{Z})}\to0,
\]
where $\Gamma_W$ is the quadratic functor of Whitehead
and the third homomorphism is the Hurewicz homomorphism.

\begin{theorem}
Let $M$ be a $PD_4$-complex with infinite fundamental group $\pi$.
Then the following are equivalent:

\begin{enumerate}
\item either $M$ is aspherical, or $\widetilde{M}\sim{S^2}$ or $S^3$;

\item $\widetilde{M}$ is homotopy finite;

\item $\pi_3(M)$ is finitely generated as an abelian group;

\item $\pi$ has finitely many ends and $\pi_2(M)$ is finitely generated 
as an abelian group.
\end{enumerate}
\end{theorem}

\begin{proof}
Clearly $(1)\Rightarrow(2)\Rightarrow(3)$ and (4).
Since $\pi$ is finitely presentable,
$E^2\mathbb{Z}$ is torsion free, and so $\Pi=\pi_2(M)$ is torsion free also.
If $\pi_3(M)$ is finitely generated as an abelian group
then  $H_3(\widetilde{M};\mathbb{Z})$ and $\Gamma_W(\Pi)$ 
are finitely generated.
Hence $\pi$ has finitely many ends and $\Pi$ is finitely generated.
Thus $(3)\Rightarrow(4)$.

If (4) holds then $\pi$ has one or two ends and $Hom(\Pi;\mathbb{Z}[\pi])=0$. 
Hence $E^2\mathbb{Z}\cong\Pi$, by the evaluation exact sequence.
If $\pi$ has one end then either $E^2\mathbb{Z}=\Pi=0$, 
in which case $M$ is aspherical, or both are infinite cyclic,
in which case $\widetilde{M}\simeq{S^2}$.
If $\pi$ has two ends we may assume without loss of generality that 
$\pi\cong\mathbb{Z}$.
But then $\Pi=0$ and $\widetilde{M}\simeq{S^3}$.
Thus $(4)\Rightarrow(1)$.
\end{proof}

In particular, if $\pi$ is infinite and either
$\pi_3(M)=0$ or $\pi$ has one end and $\pi_2(M)=0$ then $M$ is aspherical.
However, if $M=\#^rS^1\times{S^3}$ for some $r>1$
then $\pi_2(M)=0$ but $\pi_3(M)$ is not finitely generated.

\end{document}